\documentclass[12pt]{article}%
\usepackage{amsmath}
\usepackage{amsfonts}
\usepackage{mitpress}
\usepackage{amssymb}
\usepackage{color}

\usepackage{graphicx}%
\providecommand{\U}[1]{\protect\rule{.1in}{.1in}}
\newtheorem{theorem}{Theorem}

\newtheorem{definition}[theorem]{Definition}

\newtheorem{lemma}[theorem]{Lemma}

\newtheorem{proposition}[theorem]{Proposition}
\newtheorem{remark}[theorem]{Remark}

\newenvironment{proof}[1][Proof]{\noindent\textbf{#1.} }{\ \rule{0.5em}{0.5em}}
\newdimen\dummy
\dummy=\oddsidemargin
\def\ds{\displaystyle}
\newcommand{\one}{1\!\!\!\;\mathrm{l}}
\def\R{\mathbb R}
\def\N{\mathbb N}

\def\E{\mathbb E}
\def\Q{\mathbb Q}

\begin{document}

\title{Fokker-Planck equations for  SPDE with  non-trace class noise}

\author{G. Da Prato\\
{\small Scuola Normale Superiore, Pisa, Italy}
  \\ F. Flandoli
  \\ {\small    Universit\`a di Pisa, Italy }\\
 M. R\"{o}ckner
\\
{\small  University of Bielefeld, Germany } }\bigskip

 \maketitle\bigskip

 \noindent{\bf Abstract}. In this paper we develop a new technique to prove existence of solutions 
of Fokker-Planck equations on Hilbert spaces for Kolmogorov operators with non trace-class second order coefficients or equivalently with an associated stochastic partial differential equation (SPDE) with non trace-class noise. Applications include stochastic $2D$ and $3D$--Navier--Stokes equations with non trace-class additive noise. 

 \bigskip

\section{Introduction}

Our aim is to solve the infinite dimensional Fokker-Planck Equation 
\[
L^{\ast}\mu=0,\qquad\mu|_{t=0}=\mu_{0}\eqno{(FPE)}
\]
in a space of measure valued solutions of the form $\mu_{t}dt$ where $\mu_{t}$
are probability measures on $H$.  The problem has been studied intensively in recent years (see e.g, \cite{BDR}, \cite {BDR2}, \cite{BDR3}, \cite{BDRRS13}, \cite{RZZ13} and the references therein). Concerning existence there are two different approaches to infinite dimensional Fokker-Planck equations depending on whether the second order coefficient of the corresponding  Kolmogorov operator $L$ (see below) is of trace class or not. The first case is studied in detail in \cite{BDRRS13} (including also the case of   continuity equations whose second order coefficient is identically zero) and the approach is based on the method of Lyapunov functions. This method, however, so far could not be implemented when the second order coefficient is not of trace class. This case has been studied in \cite{BDR2} and \cite{RZZ13},  using an approximation technique, based on solving the stochastic differential equations associated to the approximating  Kolmogorov operator $L_n$. Then suitable accumulation points are proved to be solutions of the limiting given Fokker-Planck equation. The results in  
\cite{BDR2} and \cite{RZZ13} are, however, very limited in  applications, e.g.  essentially only Fokker-Planck equations associated to stochastic reaction--diffusion and Burgers equations  with non trace class noise (including white noise)   or combinations of such are covered.

In this paper we develop a new general  approach for this ``non-trace class case'' which applies to a much wide class of examples, including  in particular, the Fokker-Planck equation of the stochastic $2D$ and $3D$ Navier--Stokes equations with non-trace class noise (se Section 5.2 below). For the sake of simplicity we restrict ourselves to the case where the second coefficient of $L$ is 
constant.
  In the language of stochastic
equations we   restrict ourselves to additive noise.

The Kolmogorov operator \thinspace$L$ in equation (FPE) above is defined as
\[
\left(  Lu\right)  \left(  x,t\right)  :=\frac{\partial u}{\partial t}\left(
x,t\right)  +\sum_{i=1}^{\infty}a^{i}\left(  \partial_{x_{i}}^{2}u\right)
\left(  x,t\right)  +\sum_{i=1}^{\infty}b^{i}\left(  x,t\right)  \left(
\partial_{x_{i}}u\right)  \left(  x,t\right)
\]
on all functions $u:H\times\left[  0,T\right]  \rightarrow\mathbb{R}$ which
are smooth and finite dimensional (also called cylindrical). Here $H$ is a separable Hilbert space (norm
$\left\Vert .\right\Vert _{H}$, inner product $\left\langle .,.\right\rangle
_{H}$), $\left(  e_{n}\right)  $ is a c.o.s. in $H$, $H_{n}$ is the span of
$e_{1},...,e_{n}$, $\pi_{n}$ the corresponding finite dimensional projection,
and the attribute \textquotedblleft finite dimensional\textquotedblright\ to
$u$ means that $u\left(  t,x\right)  =u(\pi_nx,t)  $ for all
$x$, for some $n\in\mathbb{N}$ and $u_n:\mathbb R^n\times [0,T]\to\mathbb R$. Here and below we shall always identify $H_n$ with $\R^n$ fixing $(e_n)$.

The ideas of the present work are general, in particular the approach by an
auxiliary Fokker-Planck Equation on \textit{product space}. We develop them
under quite general assumptions which include basic cases
 like stochastic semilinear parabolic equations with linear growth (but see also Remark \ref{r14}) and, mainly, stochastic 2D and 3D Navier-Stokes
equations. A \textit{direct solution} of the Fokker--Planck equation corresponding to these
equations when the noise has the covariance of the class considered here (not
trace class or as general as possible) is new; for other approaches to the
existence of solutions for the stochastic 2D and 3D Navier-Stokes equations,
especially in the direction of general covariance, see for instance
\cite{FlaNodea}, \cite{AlbFer}, \cite{AlbFer2}, \cite{DaPDeb2D} in 2D and
\cite{DaPDeb}, \cite{FlaGat}, \cite{FlaJFA}, \cite{FlaRomTrans},
\cite{FlaRom}, \cite{FlaCIME} in 3D.

\section{Assumptions and main result\label{section assumptions}}

The numbers $a^{i}\geq0$ and the measurable functions $b^{i}:H\times\left[
0,T\right]  \rightarrow\mathbb{R}$ are subject to a series of assumptions. We
do \textit{not} assume the finite trace condition $\sum_{i=1}^{\infty}
a^{i}<\infty$ as in \cite{BDR}, \cite{BDRRS13}, but we do not allow dependence on $\left(
x,t\right)  $ (hence our applications restrict to additive noise).

Let $E\subset H$ be a separable Banach space with dense continuous injection.
Denote the norm in $E$ by $\left\Vert .\right\Vert _{E}$. We assume $e_{i}\in
E$ and other conditions below.

We assume that $b^{i}$ have the structure
\[
b^{i}\left(  x,t\right)  =-\alpha_{i}^{2}x_{i}+f^{i}\left(  x,t\right)
\]
with real numbers $\alpha_{i}^{2}>0$ and continuous functions $f^{i}
:H\times\left[  0,T\right]  \rightarrow\mathbb{R}$. The basic assumption on
the sequences $\left(  a^{i}\right)  $, $\left(  \alpha_{i}^{2}\right)  $ is
\begin{equation}
\lim_{i\rightarrow\infty}\alpha_{i}^{2}=\infty,\qquad\sum_{i=1}^{\infty}%
\frac{a^{i}}{\alpha_{i}^{2}}<\infty. \label{assumption 1}%
\end{equation}
Concerning the sequence of functions $f^{i}\left(  x,t\right)  $, we assume%
\begin{align}
\left\vert f^{i}\left(  x,t\right)  \right\vert  &  \leq C_{i}\left(
1+\left\Vert x\right\Vert _{H}^{p_{0}}\right)  ,\quad f^{i}\left(
\cdot,t\right)  \text{ locally Lipschitz in }x\text{,}\label{assumption 2}\\
\text{uniformly in }t  &  \in\left[  0,T\right] \nonumber
\end{align}
for some $C_{i}>0$, $p_{0}\geq1$.

Denote by $V$ the Hilbert space
\[
V=\left\{  x\in H:\left\Vert x\right\Vert _{V}^{2}:=\sum_{i=1}^{\infty}%
\alpha_{i}^{2}x_{i}^{2}<\infty\right\}
\]
where $x_{i}=\left\langle x,e_{i}\right\rangle _{H}$. It is compactly embedded
in $H$, by assumption (\ref{assumption 1}). Let $V^{\prime}$ be the dual space
of $V$ and let us use the identification $H=H^{\prime}$, so that  $V\subset H\subset
V^{\prime}$, with dense injections. We write $\left\langle .,.\right\rangle $
for the dual pairing between $V^{\prime}$ and $V$, so $\left\langle
x,y\right\rangle =\left\langle x,y\right\rangle _{H}$ when $x\in H$, $y\in V$. Note that then
$\pi_n$ has a natural extension from $H$ to $V'$.
We assume also that the Banach space $E\cap V$ is dense in $H$.

We assume that there is a Borel function $f:E\times\left[  0,T\right]
\rightarrow V^{\prime}$ such that $f^{i}\left(  x,t\right)  =\left\langle
f\left(  x,t\right)  ,e_{i}\right\rangle $ (in other words, we assume that the
series $f\left(  x,t\right)  :=\sum_{i=1}^{\infty}f^{i}\left(  x,t\right)
e_{i}$ converges in $V^{\prime}$ for all $x\in E$ and $t\in\left[  0,T\right]
$), and on $f$ we assume for some $C$, $k_0\in (0,\infty)$ and $\eta\in(0,1)$
\begin{equation}
\left\langle f\left(  v+z,t\right)  ,v\right\rangle \leq\eta\left\Vert
v\right\Vert _{V}^{2}+C\left\Vert v\right\Vert _{H}^{2}\left(  \left\Vert
z\right\Vert _{E}^{2}+1\right)  +C\left\Vert z\right\Vert _{E}^{k_{0}%
}+C\label{condition = assump 3}
\end{equation}
for all $z\in E$, $v\in E\cap V$. When $f$ grows more than linearly, this
assumption embodies a form of cancellation.

Let $(\beta_i(t))_{t\ge 0}$ be a sequence of independent Wiener processes on a probability space $(\Omega, \mathcal F, P)$ with normal filtration $(\mathcal F_t)_{t\ge 0}.$ Then the series of stochastic integrals, parametrized by
$\lambda\geq0$,
\begin{equation}
\label{e4'}
Z_{t}^{\lambda}:=\sum_{i=1}^{\infty}\int_{0}^{t}e^{-\left(  t-s\right)
\left(  \alpha_{i}^{2}+\lambda\right)  }\sqrt{a^{i}}d\beta_{i}\left(
s\right)  e_{i}
\end{equation}
defines a continuous Gaussian process in $H$, by assumption
(\ref{assumption 1}). Our last assumption on $E$ is that
\begin{equation}
Z_{\cdot}^{\lambda}\text{ is an }L^{2\vee k_{0}}\left(  0,T;E\right)
\text{-valued Gaussian variable}
\label{assumption 4.0}
\end{equation}
and for every $r>0$ one has
\begin{equation}
\lim_{\lambda\rightarrow\infty}P\left(  \int_{0}^{T}\left\Vert Z_{t}^{\lambda
}\right\Vert _{E}^{2}dt>r^{2}\right)  =0.
\label{assumption 4}
\end{equation}
Let   $\left(  A,D\left(  A\right)  \right)  $ and $\left(  Q,D\left(
Q\right)  \right)  $ denote  the self-adjoint linear operators
\[
D\left(  A\right)  =\left\{  x\in H:\sum_{i=1}^{\infty}\left(  \alpha_{i}%
^{2}\left\langle x,e_{i}\right\rangle _{H}\right)  ^{2}<\infty\right\}  ,\quad
Ax=-\sum_{i=1}^{\infty}\alpha_{i}^{2}\left\langle x,e_{i}\right\rangle
_{H}e_{i}%
\]%
\[
D\left(  Q\right)  =\left\{  x\in H:\sum_{i=1}^{\infty}\left(  a^{i}%
\left\langle x,e_{i}\right\rangle _{H}\right)  ^{2}<\infty\right\}  ,\quad
Qx=\sum_{i=1}^{\infty}a^{i}\left\langle x,e_{i}\right\rangle _{H}e_{i}.
\]
 Then $Z^\lambda_t,\;t\ge 0$, defined in \eqref{e4'}, can be rewritten as 
\[
Z_{t}^{\lambda}=\int_{0}^{t}e^{\left(  t-s\right)  \left(  A-\lambda\right)
}\sqrt{Q}dW_{s},\quad t\ge 0, 
\]
which is a continuous Gaussian process in $H$, with trace class covariance (by
assumption (\ref{assumption 1}))%
\[
Q_{t}^{\lambda}=\int_{0}^{t}e^{s\left(  A-\lambda\right)  }Qe^{s\left(
A-\lambda\right)  }ds,\quad t\ge 0.
\]
By assumption     (\ref{assumption 4.0}), $Z_{\cdot}^{\lambda}$ is an $L^{2}\left(
0,T;E\right)  $-valued Gaussian variable and solves
 the linear stochastic equation in $H$
\[
dZ_{t}^{\lambda}=AZ_{t}^{\lambda}dt+\sqrt{Q}dW_{t}-\lambda Z_{t}dt,\quad
Z_{0}^{\lambda}=0.
\]

\begin{remark}
\label{r0}
\em
Under a natural condition on $e^{tA},\;t\ge 0,$ \eqref{assumption 4.0} with $Z^0_t$ replacing $Z^\lambda_t$ actually implies \eqref{assumption 4}. Indeed, under the assumption that the restriction of $e^{tA},\;t\ge 0,$ is a strongly continuous semigroup on $E$, \eqref{assumption 4}  holds.
 This can be proved as follows:
 fix $i\in\mathbb N$. Then by It\^o's product rule
 $$
 e^{\lambda t}\int_0^t e^{\alpha_i^2 s} d\beta_i(s)=\int_0^t e^{\lambda s}e^{\alpha_i^2 s} d\beta_i(s)+\int_0^t\int_0^s  e^{\alpha_i^2 r} d\beta_i(r)\lambda e^{\lambda s}ds.
 $$
Hence    
$$
\begin{array}{lll}
\displaystyle \int_0^t e^{-(t-s)(\alpha_i^2 +\lambda)} d\beta_i(s)&=&\displaystyle\int_0^t  e^{-(t-s)\alpha_i^2 } d\beta_i(s)\\
\\
&&\displaystyle-\int_0^t\int_0^s  e^{-(t-r)\alpha_i^2 } d\beta_i(r)\lambda e^{-\lambda(t- s)}ds.
 \end{array}
 $$
Multipying by $\sqrt{a^i}\,e_i$ and taking summation over $i\in \mathbb N$ (here the convergence holds in $H$), we obtain
$$
\begin{array}{l}
\displaystyle Z^\lambda_t=\int_0^t e^{(t-s)A}\sqrt Q\,dW_t-
\int_0^t e^{(t-s)A}\int_0^se^{(s-r)A}\sqrt Q\,dW_r\,\lambda e^{-\lambda(t- s)}ds\\
\\
\displaystyle=Z^0_t-(1-e^{-\lambda t})\int_0^t e^{(t-s)A}Z^0_s\,\rho_\lambda(t-s)ds,
\end{array} 
$$
where $\rho_\lambda(t-s)=\one_{[0,t]}(s)(1-e^{-\lambda t})^{-1}\lambda e^{-\lambda(t- s)}$ weakly converges to the Dirac measure in $t$ as $\lambda\to\infty$.

Hence, since $Z^0_t\in E$ for $dt$--a.e. $t\in [0,T]$, it follows that
\begin{equation}
\label{e7'}
\lim_{\lambda\to\infty}\int_0^T\|Z^\lambda_t\|^2_E\,dt=0\quad P\mbox{\rm--a.s.},
\end{equation}
in particular \eqref{assumption 4}  holds.

More precisely,
$$
\begin{array}{l}
\ds  \left(\int_0^T\|Z^0_t -\int_0^t e^{(t-s)A}Z^0_s\,\rho_\lambda(t-s)ds\|^2_E\,dt\right) ^{1/2}\\
\\
\ds = \left(\int_0^T \|\int_0^t(Z_t^0-e^{(t-s)A}Z^0_s)\rho_\lambda(t-s)ds\|^2_E\,    dt \right) ^{1/2}\\
\\
\ds \le\left(\int_0^T T\int_0^t\left\| Z_t^0-e^{sA}Z^0_{t-s})\right\|^2_E\, \rho_\lambda(s)ds\,dt \right) ^{1/2}\\
\\
\ds \le  T^{1/2} \left[\int_0^T\int_0^T\left\| Z_t^0- Z^0_{t-s}\right\|^2_E \,dt \,\rho_\lambda(s)ds \right]^{1/2}\\
\\
\ds + T^{1/2} \left[\int_0^T\int_0^{T-s} \|(1-e^{sA})Z^0_t\|_E^2\,dt \,\rho_\lambda(s)ds \right]^{1/2}.
\end{array} 
$$
Since  for $P$-a.e. $\omega\in \Omega$ the first inner integral is continuous and bounded in $s\in[0,T]$, this implies \eqref{e7'}.
\end{remark} 
 
\begin{remark}\em 
The limit property in assumption (\ref{assumption 4}) is needed because of the
power 2 of $\left\Vert z\right\Vert _{E}^{2}$ in the term $C\left\Vert
v\right\Vert _{H}^{2}\left(  \left\Vert z\right\Vert _{E}^{2}+1\right)  $ of
assumption (\ref{condition = assump 3}), which is a sort of critical value
case;\ if instead we took the term \break $C\left\Vert v\right\Vert _{H}^{2}\left(  \left\Vert
z\right\Vert _{E}^{\beta}+1\right)  $ with $\beta<2$ in
(\ref{condition = assump 3}), we only need $Z_{\cdot}^{\lambda}$ being
$L^{2}\left(  0,T;E\right)  $-valued, but this is too restrictive for
applications to Navier-Stokes equations.
\end{remark}

We define $D(L)$ to be the linear space of finite dimensional regular
functions $u:H\times\left[  0,T\right]  \rightarrow\mathbb{R}$,  i.e. $u(t,x)=u_N( \langle e_1,x  \rangle,..., \langle e_N,x  \rangle,t)$ such that  $u_N\in C^{2,1}_b(\mathbb R^N\times [0,T])$ and
 $u\left(
x,T\right)  =0$. $D(L)$ is then   a (point and measure) separating class.

\begin{definition}
\label{def sol FPE}A family of Borel probability measures $\left(  \mu
_{t}\left(  dx\right)  \right)  _{t\in\left[  0,T\right]  }$ on $H$,
measurable in $t$, is a solution of the Fokker-Planck equation $(FPE)$ above if
\[
\int_{0}^{T}\int_{H}\left\Vert x\right\Vert _{H}^{p_{0}}\mu_{t}\left(
dx\right)  dt<\infty,
\]
where $p_{0}$ is given in assumption (\ref{assumption 2}), and
\[
\int_{0}^{T}\int_{H}\left(  Lu\right)  \left(  x,t\right)  \mu_{t}\left(
dx\right)  dt+\int_{H}u\left(  x,0\right)  \mu_{0}\left(  dx\right)  =0
\]
for all $u\in D\left(  L\right)  $.
\end{definition}

The double integral in the above formulation is then well defined  (see Remark 5 below) because of (2) and  the assumed moment condition. We can now state
our main theorem.

\begin{theorem}
\label{main theorem}Under the assumptions (\ref{assumption 1}%
)-(\ref{assumption 4}), for any Borel probability measure $\mu_{0}$ on $H$ such
that%
\[
\int_{H}\left\Vert x\right\Vert _{H}^{p_{1}}\mu_{0}\left(  dx\right)  <\infty
\]
for some $p_{1}>p_{0}$, equation $(FPE)$ has a solution.
\end{theorem}

In Section \ref{section examples} we shall give two examples: the case of
measurable drift of at most  linear growth   and the 2D and 3D Navier-Stokes equations.

\section{Auxiliary Fokker-Planck equation on product space}

On finite dimensional regular functions $\widetilde{u}\left(  v,z,t\right)  $,
$\widetilde{u}:H\times H\times\left[  0,T\right]  \rightarrow\mathbb{R}$, more precisely for $\widetilde u\in D(\widetilde L),$ where
$D(\widetilde L)$ is defined analogously to $D(L)$ with $H\times H$ replacing $H$,
define the auxiliary Kolmogorov operator

\begin{align*}
(\widetilde{L}\widetilde{u})  \left(  v,z,t\right)   &
:=\frac{\partial\widetilde{u}}{\partial t}\left(  v,z,t\right)  +\sum
_{i=1}^{\infty}\left(  \left(  a^{i}\partial_{z_{i}}^{2}-\left(  \alpha
_{i}^{2}+\lambda\right)  z_{i}\partial_{z_{i}}\right)  \widetilde{u}\right)
\left(  v,z,t\right) \\
&  +\sum_{i=1}^{\infty}\left(  -\alpha_{i}^{2}v_{i}+f^{i}\left(  v+z,t\right)
+\lambda z_{i}\right)  \left(  \partial_{v_{i}}\widetilde{u}\right)  \left(
v,z,t\right)  .
\end{align*}
Here $\lambda\geq0$ is a parameter. In the simplest cases (application to drifts of at most
linear growth, for instance) we could simply take $\lambda=0$, but for
some applications we need suitable values of $\lambda$. The trick of this
parameter has been introduced in \cite{CrauelFlandoli} to prove the existence
of random attractors for the Navier-Stokes equations, and here it will be used
to prove moment estimates for solutions.

\begin{definition}
\label{def sol FPE tilde}A family of Borel probability measures $\left(
\widetilde{\mu}_{t}\left(  dv,dz\right)  \right)  _{t\in\left[  0,T\right]  }$
on $H\times H$, measurable in $t$, is a solution of the Fokker-Planck equation
on product space  
\[
\widetilde{L}^{\ast}\widetilde{\mu}=0,\qquad\widetilde{\mu}|_{t=0}
=\widetilde{\mu}_{0}, \eqno{(\widetilde{FPE})}
\]
if%
\[
\int_{0}^{T}\int_{H\times H}\left(  \left\Vert v\right\Vert _{H}^{p_{0}%
}+\left\Vert z\right\Vert _{H}^{p_{0}}\right)  \widetilde{\mu}_{t}\left(
dv,dz\right)  dt<\infty,
\]
where $p_{0}$ is given in assumption (\ref{assumption 2}), and
\[
\int_{0}^{T}\int_{H\times H}\left(  \widetilde{L}\widetilde{u}\right)  \left(
v,z,t\right)  \widetilde{\mu}_{t}\left(  dv,dz\right)  dt+\int_{H\times
H}\widetilde{u}\left(  v,z,0\right)  \widetilde{\mu}_{0}\left(  dv,dz\right)
=0
\]
for all $\widetilde{u}\in D\left(  \widetilde{L}\right)  $.
\end{definition}

\begin{remark}
\em 
\label{remark well defined}The double integral in the above formulation is
well defined. Indeed, when $\widetilde{u}\in D\left(  \widetilde{L}\right)  $,
the term $\sum_{i=1}^{\infty}\left(  \left(  a^{i}\partial_{z_{i}}^{2}-\left(
\alpha_{i}^{2}+\lambda\right)  z_{i}\partial_{z_{i}}\right)  \widetilde{u}%
\right)  \left(  v,z,t\right)  $ reduces to a finite sum and we have the bound%
\[
\left\vert \sum_{i=1}^{\infty}\left(  \left(  a^{i}\partial_{z_{i}}%
^{2}-\left(  \alpha_{i}^{2}+\lambda\right)  z_{i}\partial_{z_{i}}\right)
\widetilde{u}\right)  \left(  v,z,t\right)  \right\vert \leq C+C\left\Vert
z\right\Vert _{H}%
\]
which is integrable with respect to $\widetilde{\mu}_{t}\left(  dv,dz\right)
dt$ by the integrability assumption of the definition. Similarly, the term
\[
\sum_{i=1}^{\infty}\left(  -\alpha_{i}^{2}v_{i}+f^{i}\left(  v+z,t\right)
+\lambda z_{i}\right)  \left(  \partial_{v_{i}}\widetilde{u}\right)  \left(
v,z,t\right)
\]
reduces to a finite sum and we have the bound%
\begin{align*}
& \left\vert \sum_{i=1}^{\infty}\left(  -\alpha_{i}^{2}v_{i}+f^{i}\left(
v+z,t\right)  +\lambda z_{i}\right)  \left(  \partial_{v_{i}}\widetilde{u}%
\right)  \left(  v,z,t\right)  \right\vert \\
& \leq C\left\Vert v\right\Vert _{H}+C\sum_{i=1}^{N}\left\vert f^{i}\left(
v+z,t\right)  \right\vert +C\left\Vert z\right\Vert _{H}%
\end{align*}
for some $N>0$, and thus by our main assumptions this is dominated by
\[
  C\left\Vert v\right\Vert _{H}+C_{N,p}\left(  1+\left\Vert v\right\Vert
_{H}^{p_{0}}+\left\Vert z\right\Vert _{H}^{p_{0}}\right)  +C\left\Vert
z\right\Vert _{H},
\]
which is again integrable with respect to $\widetilde{\mu}_{t}\left(
dv,dz\right)  dt$.
\end{remark}
\begin{remark}
\label{r6'}
\em 
 Let $\pi_i:H\times H\to H,\;i=1,2,$ denote the canonical projections and fix $t\in[0,T]$. Then $\widetilde\mu_t\circ \pi_2^{-1}=$ the law of $Z_t^\lambda=N_{Q_t}$,  i.e. the mean zero Gaussian measure on $H$ with covariance operator $Q_t$, where
 $$
 Q_t:=\int_0^te^{-2s(A+\lambda)}Qds.
 $$
 Indeed, it is easy to check that $(\widetilde\mu_t\circ \pi_2^{-1})(dx)dt$ solves (FPE) with $L$ replaced by
 $$
 L_0u:=\frac{\partial u}{\partial t}+\sum_{i=1}^\infty(a^i \partial^2_{z_i}-(\alpha_i^2+\lambda)z_i)\partial _{z_i }u, 
 $$
 with domain $D(L)$ (as above) and $\mu_0=\delta_0$. But for this Fokker--Planck equation $(N_{Q_t})_{t\in[0,T]}$ is the unique solution.
\end{remark}

\begin{remark}
\em 
In Definition \ref{def sol FPE tilde} it is sufficient to assume
\[
\int_{0}^{T}\int_{H\times H}\left(  \left\Vert v+z\right\Vert _{H}^{p_{0}%
}+\left\Vert v\right\Vert _{H}+\left\Vert z\right\Vert _{H}\right)
\widetilde{\mu}_{t}\left(  dv,dz\right)  dt<\infty.
\]

\end{remark}

\begin{lemma}
\label{main lemma}If $\widetilde{\mu}_{t}\left(  dv,dz\right)  dt$ is a
solution of $(\widetilde{FPE})$ on product space and if $\widetilde{\mu}_{0}$
and $\mu_{0}$ are related by the condition%
\begin{equation}
\int_{H\times H}\varphi\left(  v+z\right)  \widetilde{\mu}_{0}\left(
dv,dz\right)  =\int_{H}\varphi\left(  x\right)  \mu_{0}\left(  dx\right)
\label{i.c.}%
\end{equation}
then $\mu_{t}\left(  dx\right)  $, defined for all $t\in\left[  0,T\right]  $
as
\begin{equation}
\int_{H}\varphi\left(  x\right)  \mu_{t}\left(  dx\right)  :=\int_{H\times
H}\varphi\left(  v+z\right)  \widetilde{\mu}_{t}\left(  dv,dz\right)
,\qquad\varphi\in C_{b}\left(  H\right)  , \label{contraction}%
\end{equation}
is a weak solution of $(FPE)$.
\end{lemma}

\begin{proof}
\textbf{Step 1}.  Let $u\in D(L)$. Define 
\[
\widetilde{u}\left(  v,z,t\right)  :=u\left(  v+z,t\right)  .
\]
Then we have
\begin{align*}
\left(  Lu\right)  \left(  v+z,t\right)    & =\frac{\partial u}{\partial
t}\left(  v+z,t\right)  +\sum_{i=1}^{\infty}a^{i}\left(  \partial_{x_{i}}%
^{2}u\right)  \left(  v+z,t\right)  \\
& +\sum_{i=1}^{\infty}b^{i}\left(  v+z,t\right)  \left(  \partial_{x_{i}%
}u\right)  \left(  v+z,t\right)
\end{align*}%
\begin{align*}
(  \widetilde{L}\widetilde{u})  \left(  v,z,t\right)   &
=\frac{\partial u}{\partial t}\left(  v+z,t\right)  +\sum_{i=1}^{\infty
}\left(  \left(  a^{i}\partial_{x_{i}}^{2}-\left(  \alpha_{i}^{2}%
+\lambda\right)  z_{i}\partial_{x_{i}}\right)  u\right)  \left(  v+z,t\right)
\\
&  +\sum_{i=1}^{\infty}\left(  -\alpha_{i}^{2}v_{i}+f^{i}\left(  v+z,t\right)
+\lambda z_{i}\right)  \left(  \partial_{x_{i}}u\right)  \left(  v+z,t\right)
\end{align*}%
\begin{align*}
&  =\frac{\partial u}{\partial t}\left(  v+z,t\right)  +\sum_{i=1}^{\infty
}a^{i}\left(  \partial_{x_{i}}^{2}u\right)  \left(  v+z,t\right)  \\
&  +\sum_{i=1}^{\infty}\left(  -\alpha_{i}^{2}\left(  v_{i}+z_{i}\right)
+f^{i}\left(  v+z,t\right)  \right)  \left(  \partial_{x_{i}}u\right)  \left(
v+z,t\right).
\end{align*}
So, we deduce%
\[
\left(  Lu\right)  \left(  v+z,t\right)  =\left(  \widetilde{L}\widetilde{u}%
\right)  \left(  v,z,t\right)  .
\]

\textbf{Step 2}. The integrability condition of Definition \ref{def sol FPE}
holds, since by definition of $\mu_{t}\left(  dx\right)  $, we have
\begin{align*}
\int_{0}^{T}\int_{H}\left\Vert x\right\Vert _{H}^{p_{0}}\mu_{t}\left(
dx\right)  dt  &  =\int_{0}^{T}\int_{H\times H}\left\Vert v+z\right\Vert
_{H}^{p_{0}}\widetilde{\mu}_{t}\left(  dv,dz\right)  dt\\
&  \leq C_{p_{0}}\int_{0}^{T}\int_{H\times H}\left(  \left\Vert v\right\Vert
_{H}^{p_{0}}+\left\Vert z\right\Vert _{H}^{p_{0}}\right)  \widetilde{\mu}%
_{t}\left(  dv,dz\right)  dt<\infty.
\end{align*}

\textbf{Step 3}. By definition of $\mu_{t}\left(  dx\right)  $, we have
\[
\int_{0}^{T}\int_{H}\left(  Lu\right)  \left(  x,t\right)  \mu_{t}\left(
dx\right)  dt=\int_{0}^{T}\int_{H\times H}\left(  Lu\right)  \left(
v+z,t\right)  \widetilde{\mu}_{t}\left(  dv,dz\right)  dt.
\]
Hence, by the previous step, with $\widetilde{u}\left(  v,z,t\right)
:=u\left(  v+z,t\right)  $,%
\[
\int_{0}^{T}\int_{H}\left(  Lu\right)  \left(  x,t\right)  \mu_{t}\left(
dx\right)  dt=\int_{0}^{T}\int_{H\times H}\left(  \widetilde{L}\widetilde{u}%
\right)  \left(  v,z,t\right)  \widetilde{\mu}_{t}\left(  dv,dz\right)  dt.
\]
This  and  (\ref{i.c.}) imply the claim of the lemma.
\end{proof}

\begin{remark}
\label{r8}
\em
For a given Borel probability measure $\mu_{0}$ on $H$ it is easy to find a
Borel probability measure $\widetilde{\mu}_{0}$ on $H\times H$ such that
(\ref{i.c.}) holds. Simply, define
\[
\widetilde{\mu}_{0}\left(  dv,dz\right)  :=\varepsilon_{\left(  0,x\right)
}\left(  dv,dz\right)  \mu_{0}\left(  dx\right)
\]
where $\varepsilon_{\left(  0,x\right)  }$ is the Dirac measure in $\left(
0,x\right)  \in H\times H$. Then clearly (\ref{i.c.}) holds. Then the second marginal of $\widetilde\mu_0$ is just $\mu_0$. Another choice with first marginal equal to $\mu_0$ is $\widetilde\mu_0=\mu_0\otimes \delta_0.$ Hence any convex combination of these two satisfies \eqref{i.c.}.
\end{remark}

Thus, to prove existence of solutions of $(FPE)$, it is sufficient
to solve the auxiliary Fokker-Planck equation  $(\widetilde{FPE})$, with suitable initial condition.

\section{Existence theorem for the auxiliary equation $\protect\widetilde{(FPE)}
$}

In this section we want to prove the existence of a solution to the equation
(called above ($\widetilde{FPE}$))
\[
\int_{0}^{T}\int_{H\times H}\left(  \widetilde{L}\widetilde{u}\right)  \left(
v,z,t\right)  \widetilde{\mu}_{t}\left(  dv,dz\right)  dt+\int_{H\times
H}\widetilde{u}\left(  v,z,0\right)  \widetilde{\mu}_{0}\left(  dv,dz\right)
=0
\]
with the initial condition $\widetilde{\mu}_{0}=\mu_{0}\otimes\delta_{0}$.
This initial condition satisfies (\ref{i.c.}). One can decompose
$\widetilde{\mu}_{0}$ in other ways (see Remark \ref{r8}).

\begin{theorem}
\label{main theorem FPE tilde}  Let  the assumptions $(\ref{assumption 1}
)$-$(\ref{assumption 4})$ hold and let $\mu_{0}$ be a Borel probability measure on $H$
such that
\[
\int_{H}\left\Vert x\right\Vert _{H}^{p_{1}}\mu_{0}\left(  dx\right)  <\infty
\]
for some $p_{1}>p_{0}$. Then there exists $\lambda_{0}\geq0$ such that for
every $\lambda\geq\lambda_{0}$ equation $(\widetilde{FPE})$ has a   solution.
\end{theorem}

The proof is done in the following subsections. By Lemma \ref{main lemma}, this
proves our main Theorem \ref{main theorem}.

\subsection{A consequence of Fernique's theorem}

\begin{proposition}
\label{proposition Fernique}For every $K>0$ there is $\lambda_0>0$ such that for all $\lambda\ge \lambda_0$
\begin{equation}
E\left[  e^{\int_{0}^{T}K\left\Vert Z_{t}^{\lambda}\right\Vert _{E}^{2}%
dt}\right]  \leq e^{\frac{1}{4}}+\frac{e^{2}}{e^{2}-1}. \label{Fernique}%
\end{equation}

\end{proposition}

\begin{proof}
Since $Z_{\cdot}^{\lambda}$ is a Gaussian r.v. in the Banach space
$L^{2}\left(  0,T;E\right)  $, Fernique's theorem states that there exists
$\gamma_{0}>0$ such that $$E\left[  e^{\gamma\int_{0}^{T}\left\Vert
Z_{t}^{\lambda}\right\Vert _{E}^{2}dt}\right]  <\infty$$ for all $\gamma
\in\left(  0,\gamma_{0}\right)  $. We need a   relation between
$\lambda$ and $\gamma_{0}$, so we use the following version of Fernique's theorem
(see \cite{DaPrato-Zabczyk}): given an $L^{2}\left(  0,T;E\right)  $-valued
Gaussian variable $Z$, if two real numbers $\gamma,r>0$ satisfy
\[
\log\left(  \frac{1-P\left(  \left\Vert Z\right\Vert _{L^{2}\left(
0,T;E\right)  }\leq r\right)  }{P\left(  \left\Vert Z\right\Vert
_{L^{2}\left(  0,T;E\right)  }\leq r\right)  }\right)  +32\gamma r^{2}\leq-1,
\]
then
\[
E\left[  e^{\gamma\left\Vert Z\right\Vert _{L^{2}\left(  0,T;E\right)  }^{2}%
}\right]  \leq e^{16\gamma r^{2}}+\frac{e^{2}}{e^{2}-1}.
\]

Now, given $K>0$, choose $r=\frac{1}{8\sqrt K}$. By assumption
(\ref{assumption 4}), there exists $\lambda_0>0$ such that 
$$
P\left(  \int
_{0}^{T}\left\Vert Z_{t}^{\lambda}\right\Vert _{E}^{2}dt\leq r\right)
\geq\frac{1}{e^{-3/2}+1},\quad \forall\;\lambda\ge \lambda_0.$$ Then
\[
\log\left(  \frac{1-P\left(  \int_{0}^{T}\left\Vert Z_{t}^{\lambda}\right\Vert
_{E}^{2}dt\leq r\right)  }{P\left(  \int_{0}^{T}\left\Vert Z_{t}^{\lambda
}\right\Vert _{E}^{2}dt\leq r\right)  }\right)  \leq-\frac{3}{2}.
\]
Therefore,
\[
\log\left(  \frac{1-P\left(  \left\Vert Z_{\cdot}^{\lambda}\right\Vert
_{L^{2}\left(  0,T;E\right)  }\leq r\right)  }{P\left(  \left\Vert Z_{\cdot
}^{\lambda}\right\Vert _{L^{2}\left(  0,T;E\right)  }\leq r\right)  }\right)
+32Kr^{2}\leq-1.
\]
By the previous version of Fernique's theorem we have%
\[
E\left[  e^{K\left\Vert Z_{\cdot}^{\lambda}\right\Vert _{L^{2}\left(
0,T;E\right)  }^{2}}\right]  \leq e^{\frac{1}{4}}+\frac{e^{2}}{e^{2}-1}.
\]
The proof is complete.
\end{proof}

\subsection{Approximating problem, moment estimate}

We use the same notations and objects of the previous subsection but we
enlarge, if necessary, the filtered probability space $\left(  \Omega,
\mathcal F,(\mathcal F_{t})_{t\ge 0},P\right)  $ in such a way that there exists an $\mathcal F_{0}$-measurable r.v.
$V\left(  0\right)  $ with law $\mu_{0}$. Set
\[
A_{n}x=-\sum_{i=1}^{n}\alpha_{i}^{2}\left\langle x,e_{i}\right\rangle
_{H}e_{i},\quad f_{n}\left(  x,t\right)  =\sum_{i=1}^{n}f^{i}\left(
x,t\right)  e_{i}=\pi_{n}f\left(  x,t\right)
\]
(the latter equality is true only for $x\in E$). Consider the finite
dimensional system in $\pi_{n}\left(  H\right)  $ for the unknown
$V_{n}\left(  t\right)  $, driven by the known Gaussian process $Z_{t}%
^{\lambda}$ defined in the previous section:
\[
\frac{dV_{n}\left(  t\right)  }{dt}=A_{n}V_{n}\left(  t\right)  +f_{n}\left(
V_{n}\left(  t\right)  +Z_{t}^{\lambda},t\right)  +\lambda \pi_n(Z_{t}^{\lambda
}),\quad V_{n}\left(  0\right)  =\pi_{n}V\left(  0\right)  .
\]
This is a random differential equation. For each $n\in\mathbb{N}$ and
$\lambda>0$, as a stochastic equation, it has a unique global continuous
$\mathcal F_{t}$-adapted solution $V_{n}$ (a strong solution, in the stochastic sense).
Indeed, given any continuous path of $Z_{\cdot}^{\lambda}$,  this follows from the local Lipschitz property of each $f^{i}$, and assumption (3). (See e.g. \cite[Theorem 3.1.1]{PR}).
Whence a unique global solution is established for $P$-a.e. $\omega\in\Omega$
(those for which $Z_{\cdot}^{\lambda}\left(  \omega\right)  $ is continuous)
and it is an adapted process, by uniqueness.

Notice that $Z_{t}^{\lambda}\in E$  for $dt$-a.e. $t\in[0,T]$
and $V_{n}\left(  t\right)  \in E\cap V$
(because $\pi_{n}\left(  H\right)  \subset E\cap V$) for every $t\in\left[
0,T\right]  $, with probability one. Thus we may apply the inequalities of our
assumptions with $z=Z_{t}^{\lambda}$ and $v=V_{n}\left(  t\right)  $.

Using the notations of inner product and norm of $H$ also in $H_{n}$, from
assumption (\ref{condition = assump 3}) we have
\[
\frac{1}{2}\frac{d\left\Vert V_{n}\right\Vert _{H}^{2}}{dt}-\left\langle
A_{n}V_{n},V_{n}\right\rangle _{H}\leq\left\vert \left\langle f_{n}\left(
V_{n}+Z^{\lambda},t\right)  ,V_{n}\right\rangle _{H}\right\vert +\lambda
\left\Vert Z^{\lambda}\right\Vert _{H}\left\Vert V_{n}\right\Vert _{H}%
\]%
\begin{align*}
& \leq\eta\left\langle A_{n}V_{n},V_{n}\right\rangle _{H}+C\left\Vert
V_{n}\right\Vert _{H}^{2}\left(  \left\Vert Z^{\lambda}\right\Vert _{E}%
^{2}+1\right)  \\
& +C\left\Vert Z^{\lambda}\right\Vert _{E}^{k_{0}}+C+\lambda^{2}\left\Vert
Z^{\lambda}\right\Vert _{H}^{2}+\left\Vert V_{n}\right\Vert _{H}^{2}.
\end{align*}
Just in order to unify some expressions, and without restriction, let us
assume from now on that $k_{0}\geq2$; otherwise it is only a matter of keeping
explicitly the term$\left\Vert Z^{\lambda}\right\Vert _{H}^{2}$. With possibly changing constants $C$, we have%
\begin{equation}
\frac{d\left\Vert V_{n}\right\Vert _{H}^{2}}{dt}-\left\langle A_{n}V_{n}%
,V_{n}\right\rangle _{H}\leq C\left\Vert V_{n}\right\Vert _{H}^{2}\left(
\left\Vert Z^{\lambda}\right\Vert _{E}^{2}+1\right)  +\left(  C+\lambda
^{2}\right)  \left\Vert Z^{\lambda}\right\Vert _{E}^{k_{0}}%
+C.\label{inequality like energy}%
\end{equation}
By Gronwall's lemma we get (using also $\left\Vert V_{n}\left(  0\right)
\right\Vert ^{2}_H\leq\left\Vert V\left(  0\right)  \right\Vert ^{2}_H$)%
\[
\left\Vert V_{n}\left(  t\right)  \right\Vert _{H}^{2}\leq e^{I^{\lambda
}\left(  0,t\right)  }\left\Vert V\left(  0\right)  \right\Vert ^{2}_H+\int%
_{0}^{t}e^{I^{\lambda}\left(  s,t\right)  }\left(  \left(  C+\lambda
^{2}\right)  \left\Vert Z_{s}^{\lambda}\right\Vert _{E}^{k_{0}}+C\right)  ds,
\]
where
\[
I^{\lambda}\left(  s,t\right)  =\overline{C}\int_{s}^{t}\left(  \left\Vert
Z_{r}^{\lambda}\right\Vert _{E}^{2}+1\right)  dr
\]
and thus%
\[
\left\Vert V_{n}\left(  t\right)  \right\Vert _{H}^{2}\leq e^{I^{\lambda
}\left(  0,T\right)  }\left[  \left\Vert V\left(  0\right)  \right\Vert
^{2}_H+\int_{0}^{T}\left(  \left(  C+\lambda^{2}\right)  \left\Vert
Z_{s}^{\lambda}\right\Vert _{E}^{k_{0}}+C\right)  ds\right]  .
\]
We have denoted the constant in $I\left(  s,t\right)  $ by $\overline{C}$ to
emphasize that it is not a generic constant, but the one obtained so far in
that estimate; it does not depend on $\lambda,$ neither on $n$ and nor on $\omega$.

Notice that this inequality gives us
\begin{equation}
\left\Vert V_{n}\left(  t\right)  \right\Vert _{H}\leq e^{\frac{1}%
{2}I^{\lambda}\left(  0,T\right)  }\left[  \left\Vert V\left(  0\right)
\right\Vert _H+\left(  \int_{0}^{T}\left(  \left(  C+\lambda^{2}\right)
\left\Vert Z_{s}^{\lambda}\right\Vert _{E}^{k_{0}}+C\right)  ds\right)
^{1/2}\right]  \label{ineq 1}%
\end{equation}
so we are not limited to work in the sequel with powers of $\left\Vert
V_{n}\left(  t\right)  \right\Vert _{H}$ greater than 2.

Let $p_{1}>p_{0}$ be the value given in the assumptions of the theorem.

\begin{lemma}
\label{lemma moment estimate}For every $p\in\left(  p_{0},p_{1}\right)  $,
there exists $\lambda_{p}\geq0$ and $C_{p}>0$ such that, using the process
$Z^{\lambda_{p}}$ in the previous construction, we have%
\begin{equation}
E\left[\sup_{t\in[0,T]}  \left\Vert V_{n}\left(  t\right)  \right\Vert _{H}^{p}\right]  \leq
C_{p} \label{moment estimate}%
\end{equation}
for every $n\in\mathbb{N}$.
\end{lemma}

\begin{proof}
From inequality (\ref{ineq 1}) we deduce%
\[
\left\Vert V_{n}\left(  t\right)  \right\Vert _{H}^{p}\leq C_{p,T}e^{\frac
{p}{2}I^{\lambda}\left(  0,T\right)  }\left[  \left\Vert V\left(  0\right)
\right\Vert _H ^{p}+\left(  \int_{0}^{T}\left(  \left(  C+\lambda^{2}\right)
\left\Vert Z_{s}^{\lambda}\right\Vert _{E}^{k_{0}}+C\right)  ds\right)
^{p/2}\right]
\]
hence, for every $q,q^{\prime}\in\left(  1,\infty\right)  $ such that
$\frac{1}{q}+\frac{1}{q^{\prime}}=1$,%
\[
E\left[\sup_{t\in[0,T]} \left\Vert V_{n}\left(  t\right)  \right\Vert _{H}^{p}\right]  \leq
C_{p,q,T}E\left[  e^{\frac{pq}{2}I^{\lambda}\left(  0,T\right)  }\right]
^{1/q}E\left[  \left\Vert V\left(  0\right)  \right\Vert ^{pq^{\prime}%
}\right]  ^{1/q^{\prime}}%
\]%
\[
+C_{p,q,T}E\left[  e^{\frac{pq}{2}I^{\lambda}\left(  0,T\right)  }\right]
^{1/q}E\left[  \left(  \int_{0}^{T}\left(  \left(  C+\lambda^{2}\right)
\left\Vert Z_{s}^{\lambda}\right\Vert _{E}^{k_{0}}+C\right)  ds\right)
^{pq^{\prime}/2}\right]  ^{1/q^{\prime}}.
\]
Choose $q^{\prime}$ such that $pq^{\prime}=p_{1}$; this gives us $E\left[
\left\Vert V\left(  0\right)  \right\Vert ^{pq^{\prime}}_H\right]  <\infty$.
Choose $\lambda$, depending on $p,q,\overline{C}$, such that $E\left[
e^{\frac{pq}{2}I^{\lambda}\left(  0,T\right)  }\right]  <\infty$:\ this is
possible because of Proposition \ref{proposition Fernique}. Finally, for those
values of the parameters, the last expected value of the last inequality is
finite because $Z_{\cdot}^{\lambda}$ is a Gaussian variable in $L^{k_{0}%
}\left(  0,T;E\right)  $, see assumption (\ref{assumption 4.0}). The proof is complete.
\end{proof}

\subsection{An additional estimate}

Inequality (\ref{inequality like energy}) also gives us%
\begin{align*}
-\int_{0}^{T}\left\langle A_{n}V_{n},V_{n}\right\rangle _{H}ds  &
\leq\left\Vert V_{n}\left(  0\right)  \right\Vert _{H}^{2}+C\int_{0}%
^{T}\left\Vert V_{n}\right\Vert _{H}^{2}\left(  \left\Vert Z^{\lambda
}\right\Vert _{E}^{2}+1\right)  ds\\
&  +\left(  C+\lambda^{2}\right)  \int_{0}^{T}\left\Vert Z^{\lambda
}\right\Vert _{E}^{k_{0}}ds+CT.
\end{align*}
Recalling the definition of the Hilbert space $V$ given in the introduction,
we have proved   our additional estimate:%
\begin{align*}
\int_{0}^{T}\left\Vert V_{n}\left(  s\right)  \right\Vert _{V}^{2}ds  &
\leq\left\Vert V_{n}\left(  0\right)  \right\Vert _{H}^{2}+C\int_{0}%
^{T}\left\Vert V_{n}\right\Vert _{H}^{2}\left(  \left\Vert Z^{\lambda
}\right\Vert _{E}^{2}+1\right)  ds\\
&  +\left(  C+\lambda^{2}\right)  \int_{0}^{T}\left\Vert Z^{\lambda
}\right\Vert _{E}^{k_{0}}ds+CT.
\end{align*}
In fact, for our later purposes we may simplify it as follows:%
\begin{equation}
\label{e13}
\int_{0}^{T}\left\Vert V_{n}\left(  s\right)  \right\Vert _{V}^{2}ds\leq
C\sup_{\left[  0,T\right]  }\left\Vert V_{n}\right\Vert _{H}^{4}+ \left( \int_{0}^{T}\left\Vert Z^{\lambda}\right\Vert
_{E}^{k_{0}\vee2}ds\right)^2+C
\end{equation}
for all $n\in\N$ with a new   constant $C$, depending also on $\lambda$ and $T$, but independent of $n$ and $\omega$.

\subsection{Approximating Fokker-Planck equation}

Given $n\in\mathbb{N}$ define $D(\widetilde L_n)$ to be the span of all functions $\widetilde u_n: H_n\times H\times [0,T]\to \R$ such that
$$
\widetilde u_n(v,z,t)=\widetilde u_{n,N}(v,\langle e_1,z  \rangle,...,\langle e_N,z  \rangle, t),
$$
for some $\widetilde u_{n,N}\in C^{2,1}_b (H_n\times H_N\times[0,T])$ (i.e. $C^2$ on $H_n\times H_N$ and $C^1$ on $[0,T]$) and $\widetilde u_n(v,z,T)=0.$
Consider on $H_{n}\times H$ the Fokker-Planck
equation
$$
\begin{array}{l}
\ds\int_{0}^{T}\int_{H_{n}\times H}\left(  \widetilde{L}_{n}\widetilde{u}
_{n}\right)  \left(  v,z,t\right)  \widetilde{\mu}_{t}^{n}\left(
dv,dz\right)  dt\\
\\
\ds +\int_{H_{n}\times H}\widetilde{u}_{n}\left(  v,z,0\right)
\widetilde{\mu}_{0}^{n}\left(  dv,dz\right)  =0,\quad\forall\;\widetilde{u}
_{n}\in D(\widetilde{L}
_{n}),
\end{array}\eqno{(\widetilde{FPE^n})}
$$
where for $(v,z,t)\in H_n\times H\times [0,T]$
\begin{align*}
\left(  \widetilde{L}_{n}\widetilde{u}_{n}\right)  \left(  v,z,t\right)   &
:=\frac{\partial\widetilde{u}_{n}}{\partial t}\left(  v,z,t\right)
+\sum_{i=1}^{\infty}\left(  \left(  a^{i}\partial_{z_{i}}^{2}-\left(
\alpha_{i}^{2}+\lambda\right)  z_{i}\partial_{z_{i}}\right)  \widetilde{u}%
_{n}\right)  \left(  v,z,t\right) \\
&  +\sum_{i=1}^{n}\left(  -\alpha_{i}^{2}v_{i}+f^{i}\left(  v+z,t\right)
+\lambda z_{i}\right)  \left(  \partial_{v_{i}}\widetilde{u}_{n}\right)
\left(v,z,t\right).
\end{align*}
 The initial measure
$\widetilde{\mu}_{0}^{n}$ is, by definition, the projection on $H_{n}\times H$
of the given initial datum $\widetilde{\mu}_{0}$.

Let $\left(  V_{n},Z^{\lambda}\right)  $ be the process constructed in the
previous section. 
For $n\in \N$, $t\in[0,T]$, define $\widetilde{\mu_t^n}(dv,dz)$ to be the law of $(V_n(t), Z_t^\lambda)$. Then clearly $\widetilde{\mu_t^n}(dv,dz)dt$ solves  $(\widetilde{FPE^n})$ by It\^o's formula. Replacing $\widetilde{u}_n(v,z,t)$ in $(\widetilde{FPE^n})$ by $\varphi(t)\widetilde{u}_n(v,z,t)$ for $\varphi\in C_0([0,T))$ we easily see that $(\widetilde{FPE^n})$ is equivalent to
$$
\begin{array}{l}
\ds \int_{H_n\times H}u(v,z,t)\widetilde{\mu}^{n}_t(dv,dz)=\int_{H_n\times H}u(v,z,0)\widetilde{\mu}^{n}_0(dv,dz)\\
\\
\ds+\int_0^t\int_{H_n\times H} \widetilde{L}_{n}u(v,z,s)\widetilde{\mu}^{n}_s(dv,dz)ds,\quad \forall\;t\in[0,T],\;u\in D(\widetilde{L}^{n}).
\end{array} 
$$   
(see \cite[Remark 1.2]{BDR11} for details.)
Hence an easy consideration shows that the above equation also holds
for all $\varphi\in \mathcal F C^2_b$, i.e. all functions $\varphi: H\times H\to \R$ of the form
$$
\varphi(v,z)=\varphi_N(\langle e_1,v  \rangle,...,\langle e_N,v   \rangle,\langle e_1,z  \rangle,...,\langle e_N,z  \rangle),
$$
where $N\in\N$, $\varphi_N\in C^2_b(\R^N\times\R^N)$. Hence it follows from Remark \ref{remark well defined} and Lemma \ref{lemma moment estimate} that for every $\varphi\in \mathcal F C^2_b$ the $\R$-valued maps   
$$
t\to \widetilde{\mu}^{n}_t(\varphi):=\int_{H_n\times H}\varphi(v,z)\,\widetilde{\mu}^{n}_t(dv,dz),\quad n\in\N,
$$
are equicontinuous on $[0,T].$    

\subsection{Passage to the limit}

{\bf Step 1}. Convergence of a subsequence of $\widetilde{\mu}^{n}_t$, $n\in\N,$ for all $t\in[0,T].$ \medskip

First we extend $\widetilde{\mu}^{n}_t$ from a measure on $H_n\times H$ to a measure on $H\times H$ as follows. Let $H_n^\perp$ be the orthogonal complement of $H_n$ in $H$, $\delta_0$ the Dirac measure on $H_n^\perp$  with mass at $0\in H_n^\perp$ and $\Lambda: H_n\times H_n^\perp\times H \to H\times H$ defined by
$$
\Lambda((v_n,v_n^\perp,z))=(v_n+v_n^\perp,z),\quad v_n\in H_n,\;v_n^\perp\in H_n^\perp,\;z\in H.
$$
Then the image under $\Lambda$ of the measure $\widetilde{\mu}^{n}_t(dv_n,dz)\otimes \delta_0(dv_n^\perp)$ extends $\widetilde{\mu}_t^{n}$ to $H\times H$. Let us also denote this extension by $\widetilde{\mu}_t^{n}$. Then we have for any integrable function $g:H\times H\to \R$
\begin{equation}
\label{e16r}
\begin{array}{l}
\ds \int_{H\times H}g(v,z)\widetilde{\mu}_t^{n}(dv,dz)\\
\\
\ds=
\int_{H_n\times H_n^\perp\times  H}g(v_n+v_n^\perp,z)\widetilde{\mu}_t^{n}(dv_n,dz)
\otimes \delta_0(dv_n^\perp)\\
\\
\ds =
\int_{H_n\times  H}g(v_n,z)\widetilde{\mu}_t^{n}(dv_n,dz)
\otimes \delta_0(dv_n^\perp)\\
\\
\ds=\int_{H\times H}g(\pi_n(v),z)\widetilde{\mu}_t^{n}(dv,dz).
\end{array} 
\end{equation}
Furthermore,  by Remark \ref{r6'} and Lemma \ref{lemma moment estimate} for each $t\in [0,T]$
\begin{equation}
\label{e17r}
\sup_{n\in\N}\int_{H\times H}(\|v\|_H^p+\|z\|_H^2)\widetilde{\mu}_t^{n}(dv,dz)<+\infty,
\end{equation}
for any $p$ given in Lemma \ref{lemma moment estimate}.
Closed balls in $H\times H$ are compact and metrizable with respect to the weak topology $\tau_w$. Hence by \cite[Theorem 8.6.7]{B07}
and a diagonal argument we can find a subsequence $\widetilde{\mu}_t^{n_k},\;k\in\N$, such that $(\widetilde{\mu}_t^{n_k})$ converges $\tau_w$-weakly to a probability measure $\widetilde{\mu}_t$ on $H\times H$ as $k\to\infty$ for all $t\in [0,T]\cap\Q$. Now let $t\in[0,T]\setminus \Q$. We claim that also for such $t$ the sequence $(\widetilde{\mu}_t^{n_k})$ converges $\tau_w$-weakly to some probability measure $\widetilde{\mu}_t$ on $H\times H$. Since by \eqref{e17r} also $(\widetilde{\mu}_t^{n_k})$ has $\tau_w$-convergent subsequences, we only have to identify the limit points. So,  let $n_{k_l},\;l\in\N,$
be a subsequence such that $(\widetilde{\mu}_t^{n_{k_l}})$ $\tau_w$-weakly  converges to some probability measure $\widetilde{\nu}_t$ on $H\times H$ as $l\to\infty$. Then by the equicontinuity proved in Section 4.4 we have for all $\varphi\in \mathcal FC^2_b$ (which are all weakly continuous)
$$
\begin{array}{l}
\ds\int_{H\times H}\varphi \,d\widetilde{\nu}_t=\lim_{l\to\infty}\int_{H\times H}\varphi \,d\widetilde{\mu}^{n_{k_l}}_t\\
\\
\ds
=\lim_{l\to\infty}\lim_{s\to t,\;s\in\Q}\int_{H\times H}\varphi \,d\widetilde{\mu}^{n_{k_l}}_s= \lim_{s\to t,\;s\in\Q}\int_{H\times H}\varphi \,d\widetilde{\mu}_s.
\end{array} 
$$
Since $\mathcal FC^2_b$ is a measure separating class, this proves our claim.\medskip

{\bf Step 2}. Convergence of a subsequence of $\widetilde{\mu}_t^n(dv,dz)dt$.\medskip

Let $(\gamma_i)\subset [1,\infty)$ such that
$$
\gamma_i\uparrow\infty,\quad\sum_{i=1}^\infty \gamma_i\frac{a_i}{\alpha_i^2}<\infty.
$$
Define the operator $\Gamma$ on $H$ by
$$
\left\{\begin{array}{l}
\ds\Gamma x:=\sum_{i=1}^\infty\gamma_i^{1/2} \langle x,e_i  \rangle e_i,\\
\\
\ds D(\Gamma):=\left\{x\in H:\;\sum_{i=1}^\infty\gamma_i  \langle x,e_i  \rangle^2<\infty   \right\}   
\end{array}\right. 
$$
Then $\Gamma$ has compact level sets and furthermore
$$
\begin{array}{lll}
\ds \E[\|\Gamma(Z_t^\lambda)\|_H^2]&=&\ds\sum_{i=1}^\infty\gamma_ia^i\int_0^te^{-2(t-s)(\alpha_i^2+\lambda)} ds \\
\\
&=&\ds\frac12\sum_{i=1}^\infty \gamma_i\,\frac{a^i}{\alpha_i^2+\lambda}(1-e^{-2t(\alpha_i^2+\lambda)})\\
\\
&\le&\ds\frac12\sum_{i=1}^\infty \gamma_i\,\frac{a^i}{\alpha_i^2}=:C_\gamma<\infty.
\end{array} 
$$
Hence it follows from \eqref{assumption 4.0}, \eqref{e13} and Lemma \ref{lemma moment estimate}
$$
\sup_{n\in\N}\int_0^T\int_{H\times H}(\|v\|_V^2+\|\Gamma z\|_H^2)\,\widetilde{\mu}^n_t(dv,dz)dt<\infty.
$$
But the function $(v,z)\to \|v\|_V^2+\|\Gamma v\|_H^2$
has compact level sets on $H\times H$, hence (selecting a  subsequence if necessary) $\widetilde{\mu}^{n_k}_t(dv,dz)dt$, where $n_k,\;k\in\N,$ is as in Step 1, weakly converges to a finite measure
$\widetilde{\mu} (dt,dv,dz)$ on $[0,T]\times H\times H$ as $k\to\infty.$ But for $u\in D(\widetilde{L})$ we then have by Lebesgue's dominated convergence theorem
$$
\begin{array}{l}
\ds \int_0^T\int_{H\times H}u(v,z,t)\widetilde{\mu}_t(dv,dz)dt=\lim_{k\to\infty}\int_0^T\int_{H\times H}u(v,z,t)\,\widetilde{\mu}^{n_k}_t(dv,dz)dt\\
\\
\ds=\int_0^T\int_{H\times H}u(v,z,t)\widetilde{\mu}(dt,dv,dz).
\end{array} 
$$
Since $D(\widetilde{L})$ is a measure separating class, it follows that $\widetilde{\mu}(dt,dv,dz)=\widetilde{\mu}_t(dv,dz)dt$. So
$\widetilde{\mu}^{n_k}_t(dv,dz)dt\to \widetilde{\mu}_t(dv,dz)dt$
weakly on $H\times H\times [0,T]$ as $k\to\infty.$\medskip

\textbf{Step 3}. Passage to the limit. \medskip

Just to simplify notations, assume
that the whole sequence \break $(\widetilde{\mu}^{n}_t(dv,dz)dt)_{n\in\N}$
weakly converges to $\widetilde{\mu}_t(dv,dz)dt$ on $[0,T]\times H\times H$.
  We have to prove
that $\widetilde{\mu}_{t}\left(dv,dz\right)  dt$ is a solution of 
$(\widetilde{FPE})$. Since we have $\widetilde{\mu}^{n}_0\to \widetilde{\mu}_0$ weakly on $H\times H$ we only need to prove that
\begin{align*}
&  \lim_{n\rightarrow\infty}\int_{0}^{T}\int_{H\times H}(
\widetilde{L}_n\widetilde{u} )  \left(  v,z,t\right)  \widetilde{\mu
}_{t}^{n}\left(  dv,dz\right)  dt\\
&  =\int_{0}^{T}\int_{H\times H} (  \widetilde{L}\widetilde{u} )
\left(  v,z,t\right)  \widetilde{\mu}_{t}\left(  dv,dz\right)  dt
\end{align*}
for all $\widetilde{u}\in D(\widetilde L)$. Below we fix such a $\widetilde{u}$. Let $m\in\N$ such that
$$
\widetilde{u}(v,z,t)=\widetilde{u}(\pi_m v,\pi_mz,t)
$$
 and let $n\ge m$.
By \eqref{e16r}  the above equation follows from
\begin{equation}
\label{e18r}  
\begin{array}{l}
\ds\lim_{n\to\infty,n\ge m}\int_{0}^{T}\int_{H\times H}\psi(v,z,t)\widetilde{\mu
}_{t}^{n}(dv,dz)dt\\
\\
\ds =\int_{0}^{T}\int_{H\times H}\psi(v,z,t)\widetilde{\mu
}_{t}(dv,dz)dt,
\end{array}
\end{equation}
where
$$
\begin{array}{lll}
\ds \psi(v,z,t)&=&\ds-\sum_{i=1}^m(\alpha_i^2+\lambda)z_i\partial_{z_i}\widetilde{u}(\pi_m v,\pi_m z,t)\\
\\
&&\ds+\sum_{i=1}^m(-\alpha_i^2v_i+\lambda z_i)\partial_{v_i}\widetilde{u}(\pi_m v,\pi_m z,t)\\
\\
&&\ds+\sum_{i=1}^mf^i(v+z,t)\partial_{v_i}\widetilde{u}(v, z,t).
\end{array} 
$$
Note that  $\psi$ is a continuous function, but not bounded. So, we cannot pass to the limit in \eqref{e18r}  just by the weak convergence of 
$\widetilde{\mu}_{t}^{n}dt$ to $\widetilde{\mu}_{t}dt$ on $H\times H\times [0,T]$ proved in Step 2 above. But we can argue similarly as in the proof of Vitali's theorem. Let us give the details.

By assumption \eqref{assumption 2} we have
\begin{equation}
\label{e19r}
|\psi(v,z,t)|\le c(1+\|v\|^{p_0}_H+\|z\|^{p_0}_H).
\end{equation}
For $R\in(0,\infty)$ we define
$$
\psi_R:=\psi\wedge R\vee (-R),
$$
then
$$
\psi=\psi_R+(\psi-\psi_R),
$$
where
$$|\psi-\psi_R|\le \one_{\{|\psi|\ge R\}}2|\psi|.$$
Hence by Step 2 it suffices to show that
\begin{equation}
\label{e20r}
\limsup_{R\to\infty}\sup_{n\in\N}\int_{0}^{T}\int_{H\times H}\one_{\{|\psi|\ge R\}}|\psi|\,d\widetilde{\mu}^n_{t}\,dt=0
\end{equation}
and
\begin{equation}
\label{e21r}
\limsup_{R\to\infty} \int_{0}^{T}\int_{H\times H}\one_{\{|\psi|\ge R\}}|\psi|\,d\widetilde{\mu}_{t}\,dt=0.
\end{equation}
But we have by the H\"older inequality for every $\delta\in(0,\infty)$
$$
\begin{array}{lll}
\ds\int_{0}^{T}\int_{H\times H}\one_{\{|\psi|\ge R\}}|\psi|\,d\widetilde{\mu}^n_{t}\,dt&\le&\ds \left( \int_{0}^{T}\int_{H\times H} |\psi|^{1+\delta}\,d\widetilde{\mu}^n_{t}\,dt  \right)^{\frac1{1+\delta}}\\
\\
\ds &\times&\ds\left(\frac1R \int_{0}^{T}\int_{H\times H} |\psi|\,d\widetilde{\mu}^n_{t}\,dt  \right)^{\frac\delta{1+\delta}}.
\end{array} 
$$
By Lemma \ref{lemma moment estimate} and \eqref{e19r} this implies \eqref{e20r}.
Note that by Step 2 and Lemma \ref{lemma moment estimate} it follows by lower continuity that
$$
\int_0^T\int_{H\times H}(\|v\|_H^p+\|z\|_H^p)\widetilde{\mu}_t(dv,dz)dt<\infty,
$$
for all $p\in (p_0,p_1)$. Hence \eqref{e21r} follows similarly (even easier) as \eqref{e20r} above. \hfill $\Box$

\section{Examples\label{section examples}}

\subsection{Measurable linear growth drift}

Let $f:H\times\left[  0,T\right]  \rightarrow H$ be a measurable map such that
$$\|f(x,t)\|_H\le C(1+\|x\|_H),\quad t\in[0,T],\;x\in H
$$
for some constant $C>0$. Denote by $f^{i}\left(  t,x\right)  $ its components.
Assume that (\ref{assumption 1}) holds. Then, with $E=H$, also the other
assumptions hold. The proof of assumptions (\ref{assumption 2}) and
(\ref{assumption 4.0}) is elementary. We have
$$
\langle f(v+z,t),v\rangle_H\le C(1+\|v+z\|_H)\,\|v\|_H\le C' \|v\|_H^2+C'\|z\|_H^2+C''
$$
which implies (\ref{condition = assump 3}).  Finally,
\[
 E[\|Z_{t}^{\lambda}\|_H^2]   =\sum
_{i=1}^{\infty}\left(  1-e^{-2t\left(  \alpha_{i}^{2}+\lambda\right)
}\right)  \frac{a^{i}}{2\left(  \alpha_{i}^{2}+\lambda\right)  }\leq\sum
_{i=1}^{\infty}\frac{a^{i}}{2\left(  \alpha_{i}^{2}+\lambda\right)  }\quad .
\]%
Hence
\begin{align*}
P\left(  \int_{0}^{T}\left\Vert Z_{t}^{\lambda}\right\Vert _{E}^{2}
dt>r^{2}\right)   &  \leq r^{-2}E\left[\int_0^T\|Z_{t}^{\lambda}\|_H^2\,dt\right] =r^{-2}\int_{0}^{T}E\left[ \|Z_{t}^{\lambda}\|_H^2 \right]  dt\\
&  =\frac{1}{2}r^{-2}T\sum_{i=1}^{\infty}\frac{a^{i}}{\alpha_{i}^{2}+\lambda}
\end{align*}
which implies (\ref{assumption 4}). Note that  (4)  follows from (1). Hence Theorem \ref{main theorem} holds.

\begin{remark}
\label{r14}
\em  Similarly as in   \cite{RZZ13} we can also treat stochastic partial differential equations
on $H=L^2(0,1)$ whose drift is the sum of the Dirichlet Laplacian, a reaction--diffusion
and a  Burgers type part. However, in contrast to  \cite{RZZ13} (and also \cite{BDR2}, \cite{BDR11})  we need to assume that the reaction part is of at most quadratic growth. The details are straight forward.

\end{remark}

\subsection{Navier-Stokes equations}

Consider the stochastic Navier-Stokes equations
\begin{align*}
du+\left(  u\cdot\nabla u+\nabla p\right)  dt &  =\nu\Delta udt+\sum
_{i=1}^{\infty}\sqrt{a^{i}}e_{i}d\beta_{i}\left(  s\right)  \\
\operatorname{div}u &  =0\\
u|_{t=0} &  =u_{0}%
\end{align*}
on the torus $D=\left[  0,2\pi\right]  ^{d}$, $d=2,3$, with periodic boundary conditions.

We introduce the Hilbert space $H$ defined as the closure in $L^{2}\left(
D,\mathbb{R}^{d}\right)  $ of the set of all $\varphi\in C^{\infty}\left(
D,\mathbb{R}^{d}\right)  $ which satisfy the periodic boundary conditions and
$\operatorname{div}\varphi=0$; $H$ is a closed strict subspace of
$L^{2}\left(  D,\mathbb{R}^{d}\right)  $ and we shall denote the orthogonal
projection from $L^{2}\left(  D,\mathbb{R}^{d}\right)  $ to $H$ by $P_{H}$. We
assume $u_{0}\in H$. We introduce also the Hilbert space $V$ of all periodic
$\varphi\in H^{1}\left(  D,\mathbb{R}^{d}\right)  $ such that
$\operatorname{div}\varphi=0$; and $D\left(  A\right)  =H^{2}\left(
D,\mathbb{R}^{d}\right)  \cap V$. Then we introduce the so called Stokes
operator $A:D\left(  A\right)  \subset H\rightarrow H$ defined as
$A\varphi=P_{H}\left(  \nu\Delta\varphi\right)  $ (in fact, in the case of
periodic boundary conditions, one can show that $A\varphi=\nu\Delta\varphi$).
Since $A^{-1}$ is compact,  there exists  a complete
orthonormal system $\{e_i\}$  of eigenvectors of $A$, with eigenvalues $\left\{
-\alpha_{i}^{2}\right\}  $, that we order such that $0<\alpha_{1}^{2}%
\leq\alpha_{2}^{2}\leq...$ One can show that, with these new concepts and
notations, the space $V$ defined in Section \ref{section assumptions} and the
space $V$ defined here coincide.

Let $B\left(  .,.\right)  :D\left(  A\right)  \times V\rightarrow H$ be
defined as%
\begin{equation}
B\left(  \varphi,\psi\right)  =-P_{H}\left(  \varphi\cdot\nabla\psi\right)  .
\label{bilinear map}%
\end{equation}
The expression $\int_{D}B\left(  \varphi,\psi\right)  \left(  x\right)
\theta\left(  x\right)  dx$, $\varphi\in D\left(  A\right)  $, $\psi\in V$,
$\theta\in H$, extends to $\varphi,\psi,\theta\in V$, and several other
classes of functions. For smooth fields $\varphi$, $\psi$, $\theta\in H$ we
have%
\[
\left\langle B\left(  \varphi,\psi\right)  ,\theta\right\rangle =-\left\langle
B\left(  \varphi,\theta\right)  ,\psi\right\rangle
\]
by a simple integration by parts, and this identity extends by density to
several spaces of weaker fields.

Using the previous set-up we may formally write the stochastic Navier-Stokes
equations in abstract form (the pressure disappears since $P_{H}\left(  \nabla
p\right)  =0$):%
\[
du=\left(  Au+B\left(  u,u\right)  \right)  dt+\sum_{i=1}^{\infty}\sqrt{a^{i}%
}e_{i}d\beta_{i}\left(  s\right)  .
\]
See \cite{FlaCIME} for a review on the 3D stochastic Navier-Stokes equations
and further details on the set-up.

To connect this equation with the abstract framework of this paper we consider
the space $E=C_{per}\left(  D;\mathbb{R}^{d}\right)  \cap H$, where
$C_{per}\left(  D;\mathbb{R}^{d}\right)  $ is the space of periodic continuous
vector fields on $D$, introduce the functions $f^{i}:\left[  0,T\right]
\times H\rightarrow\mathbb{R}$ defined as
\[
f^{i}\left(  t,x\right)  =-\left\langle B\left(  x,e_{i}\right)
,x\right\rangle =\int_{D}\left(  x\left(  \xi\right)  \cdot\nabla\right)
e_{i}\left(  \xi\right)  \cdot x\left(  \xi\right)  d\xi
\]
and consider the sequences $\left\{  \alpha_{i}^{2}\right\}  $ and $\left\{
a_{i}\right\}  $ above. As remarked above, since $\int_{D}B\left(
\varphi,\psi\right)  \left(  x\right)  \theta\left(  x\right)  dx$ extends to
$\varphi,\psi,\theta\in V$, there exists $f\left(  t,x\right)  $ with values
in $V^{\prime}$ such that $f^{i}\left(  x,t\right)  =\left\langle f\left(
x,t\right)  ,e_{i}\right\rangle $;\ it is given by $B\left(  x,x\right)  $,
when $x\in D\left(  A\right)  $.

We assume that $a_{i}$ has the form%
\[
a^{i}=\alpha_{i}^{-\varepsilon}%
\]
for some $\varepsilon$ such that%
\begin{align*}
\varepsilon &  >0\text{,\qquad for }d=2\\
\varepsilon &  >1\text{,\qquad for }d=3.
\end{align*}
This guarantees assumption (\ref{assumption 1}). Indeed, on the torus $D$, the
family of eigenvectors $\left\{  e_{i}\right\}  _{i\in\mathbb{N}}$ of $A$ can
be written (see \cite{Tem}) in the form $\left\{  e_{\alpha,k}\right\}  $ with
$k\in\mathbb{Z}_{\ast}^{d}=\mathbb{Z}^{d}\backslash\left\{  0\right\}  $ and
$\alpha$ which varies in the finite set $\left\{  1,...,d-1\right\}  $ and
their associated eigenvalues, indexed in the form $\left\{  \alpha_{k}%
^{2}\right\}  _{k\in\mathbb{Z}_{\ast}^{d}}$ (for each $k\in\mathbb{Z}_{\ast
}^{d}$ the eigenvalue $\alpha_{k}^{2}$ has multiplicity $d-1$), are given by
$\alpha_{k}^{2}=\left\Vert k\right\Vert ^{2}$. If we use the complex valued
notation, one has $e_{\alpha,k}\left(  \xi\right)  =c_{\alpha,k}e^{ik\cdot\xi
}$ where, for each $k$, the set of vectors $\left\{  c_{\alpha}\right\}
_{\alpha\in\left\{  1,...,d-1\right\}  }$ is an orthonormal basis of the space
in $\mathbb{R}^{d}$ orthogonal to $k$.

Hence we may rewrite $\sum_{i=1}^{\infty}\frac{a^{i}}{\alpha_{i}^{2}}$ as%
\[
\sum_{k\in\mathbb{Z}_{\ast}^{d}}\left\Vert k\right\Vert ^{-2-\varepsilon}%
\]
and this series converge in $d=2$ for every $\varepsilon>0$, and in $d=3$ for
every $\varepsilon>1$.

We claim that under these conditions all the assumptions of the paper are
verified and thus Theorem \ref{main theorem} holds. Let us check the assumptions.

The eigenvectors $e_{i}$ and their derivatives are bounded functions and thus assumption
(\ref{assumption 2}) holds with $p_{0}=2$.

We have, for smooth fields $v,z$,
$$
\begin{array}{l}
\ds \langle  f(v+z,t),v\rangle=\int_D((v+z)\cdot\nabla)(v+z) \cdot v\,d\xi=\int_D((v+z)\cdot\nabla)z \cdot v\,d\xi\\
\\
\ds=-\int_D((v+z)\cdot\nabla)v \cdot z\,d\xi\le \eta\|v\|^2_{H^1(D)}+C_\eta\||z|\,|v+z|\|^2_{L^2(D)}\\
\\
\ds\le \eta\|v\|^2_{H^1(D)}+2C_\eta\|z\|^2_{L^\infty(D)}\;\|z\|^2_{L^2(D)}+ 2C_\eta\|z\|^2_{L^\infty(D)}\;\|v\|^2_{L^2(D)}
\end{array}
$$
and the inequality extends to all $z\in E$, $v\in E\cap V$. Thus
(\ref{condition = assump 3}) holds true, with $k_0=4$.

Finally, assumptions (\ref{assumption 4.0})-(\ref{assumption 4}) are true
under the condition imposed above on $\left\{  \alpha_{i}^{2}\right\}  $ and
$\left\{  a^{i}\right\}  $. To show this, we have to use the theory of Section 5.5.1 of
\cite{DaPrato-Zabczyk} and the explicit form of the eigenfunctions $e_{i}$ of
the Stokes operator $A$. Since we need the bounds of this reference with a
precise control of the constants, we have repeated some of the computations in
the next lemma. 

\begin{lemma}
Assume $a^{i}=\alpha_{i}^{-\varepsilon}$ with $\varepsilon$ as above. Then the
random field%
\[
Z_{t}^{\lambda}\left(  \xi\right)  =\sum_{i=1}^{\infty}\int_{0}^{t}e^{-\left(
t-s\right)  \left(  \alpha_{i}^{2}+\lambda\right)  }\sqrt{a^{i}}d\beta
_{i}\left(  s\right)  e_{i}\left(  \xi\right)
\]
has a continuous modification in $\left(  t,\xi\right)  $ and satisfies
assumptions (\ref{assumption 4.0})-(\ref{assumption 4}).
\end{lemma}

\begin{proof}
The eigenfunctions $e_{i}$ have the properties $e_{i}\in C^1\left(
D;\mathbb{R}^{d}\right)  $, $\left\vert e_{i}\left(  \xi\right)  \right\vert
\leq C$, $\left\vert \nabla e_{i}\left(  \xi\right)  \right\vert \leq
C\alpha_{i}$, required in Section 5.5.1 of \cite{DaPrato-Zabczyk}. We have
also the other property asked in that reference, namely $\sum_{i=1}^{\infty
}\frac{a^{i}}{\left(  \alpha_{i}^{2}\right)  ^{1-\delta}}<\infty$ for some
$\delta>0$, and precisely, for the sequel, we take $\delta=\frac\varepsilon4$, then \[
\sum_{i=1}^{\infty}\frac{a^{i}}{\left(  \alpha_{i}^{2}\right)  ^{1-\frac
{\varepsilon}{4}}}<\infty.
\]
Indeed, the previous series is equal to $\sum_{k\in\mathbb{Z}_{\ast}^{d}%
}\left\Vert k\right\Vert ^{-2-\frac{\varepsilon}{2}}<\infty$ with the
equivalent language of the indexes $k\in\mathbb{Z}_{\ast}^{d}$. Hence Theorem
5.20 of \cite{DaPrato-Zabczyk} applies and gives us the existence of a
continuous modification of $Z_{t}^{\lambda}\left(  \xi\right)  $ in $\left(
t,\xi\right)  $. Let us be more precise from the quantitative viewpoint. In
the sequel, we write $Z_{t}^{\lambda}\left(  \xi\right)  $ for each one of its
$d$ components, for notational simplicity. It is sufficient to prove
assumptions (\ref{assumption 4.0})-(\ref{assumption 4}) for each component of
$Z_{t}^{\lambda}\left(  \xi\right)  $.

Lemma 5.19 of \cite{DaPrato-Zabczyk} gives us the estimate%
$$
E\left[  \left\vert Z_{t}^{\lambda}\left(  \xi\right)  -Z_{t}^{\lambda}\left(
\xi^{\prime}\right)  \right\vert ^{2}\right]     \leq C_{1}\left(
\lambda\right)  \left\vert \xi-\xi^{\prime}\right\vert ^{\frac{\varepsilon}{4}},\quad \forall\;t\in [0,\infty),
$$
where
$$
C_{1}\left(  \lambda\right)    :=C_{1}\sum_{i=1}^{\infty}\frac{a^{i}%
\alpha_{i}^{\frac{\varepsilon}{4}}}{\alpha_{i}^{2}+\lambda}%
$$
for some constant $C_{1}>0$. Indeed%
\begin{align*}
& E\left[  \left\vert Z_{t}^{\lambda}\left(  \xi\right)  -Z_{t}^{\lambda
}\left(  \xi^{\prime}\right)  \right\vert ^{2}\right]  \\
& =E\left[  \left\vert \sum_{i=1}^{\infty}\int_{0}^{t}e^{-\left(  t-s\right)
\left(  \alpha_{i}^{2}+\lambda\right)  }\sqrt{a^{i}}d\beta_{i}\left(
s\right)  \left(  e_{i}\left(  \xi\right)  -e_{i}\left(  \xi^{\prime}\right)
\right)  \right\vert ^{2}\right]  \\
& =\sum_{i=1}^{\infty}\left\vert e_{i}\left(  \xi\right)  -e_{i}\left(
\xi^{\prime}\right)  \right\vert ^{2}\int_{0}^{t}e^{-2\left(  t-s\right)
\left(  \alpha_{i}^{2}+\lambda\right)  }a^{i}ds\\
& \leq C\sum_{i=1}^{\infty}\frac{a^{i}\alpha_{i}^{\frac{\varepsilon}{4}}%
}{\alpha_{i}^{2}+\lambda}\left\vert \xi-\xi^{\prime}\right\vert ^{\frac
{\varepsilon}{4}}%
\end{align*}
because%
\[
\left\vert e_{i}\left(  \xi\right)  -e_{i}\left(  \xi^{\prime}\right)
\right\vert \leq C\alpha_{i}^{\frac{\varepsilon}{4}}\left\vert \xi-\xi
^{\prime}\right\vert ^{\frac{\varepsilon}{4}}.
\]
Then
$$
E\left[  \left\vert Z_{t}^{\lambda}\left(  \xi\right)  -Z_{t}^{\lambda}\left(
\xi^{\prime}\right)  \right\vert ^{2m}\right]     \leq C_{m}\left(
\lambda\right)  \left\vert \xi-\xi^{\prime}\right\vert ^{\frac{\varepsilon
m}{4}},\quad \forall\;t\in [0,\infty),
$$
where
$$
C_{m}\left(  \lambda\right)     :=C_{m}C_{1}\left(  \lambda\right)  ^{m}.
$$
for some constant $C_{m}>0$. We remark also the easier estimate%
$$
E\left[  \left\vert Z_{t}^{\lambda}\left(  \xi\right)  \right\vert
^{2m}\right]     \leq\widetilde{C}_{m}\left(  \lambda\right)  :=\widetilde{C}%
_{m}\widetilde{C}_{1}\left(  \lambda\right)  ^{m},\quad \forall\;t\in [0,\infty),
$$
where
$$
\widetilde{C}_{1}\left(  \lambda\right)    =\widetilde{C}_{1}\sum
_{i=1}^{\infty}\frac{a^{i}}{\alpha_{i}^{2}+\lambda}$$
for some constant $\widetilde{C}_{1}>0$.

Given $\alpha\in\left(  0,1\right)  $, for the $W^{\alpha,2m}\left(  D\right)
$-norm, we have the estimate%
\begin{align*}
& E\left[  \int_{D}\left\vert Z_{t}^{\lambda}\left(  \xi\right)  \right\vert
^{2m}d\xi\right]  +E\left[  \int_{D}\int_{D}\frac{\left\vert Z_{t}^{\lambda
}\left(  \xi\right)  -Z_{t}^{\lambda}\left(  \xi^{\prime}\right)  \right\vert
^{2m}}{\left\vert \xi-\xi^{\prime}\right\vert ^{d+2m\alpha}}d\xi d\xi^{\prime
}\right]  \\
& \leq (2\pi)^d\widetilde{C}_{m}\left(  \lambda\right)  +C_{m}\left(  \lambda\right)
\int_{D}\int_{D}\left\vert \xi-\xi^{\prime}\right\vert ^{\frac{\varepsilon
m}{4}-d-2m\alpha}d\xi d\xi^{\prime}\\
& =:\widetilde{C}_{m}\left(  \lambda\right)  +C_{m}\left(  \lambda\right)
\cdot C_{m,\alpha,\varepsilon,d}%
\end{align*}
and $C_{m,\alpha,\varepsilon,d}<\infty$ if $\frac{\varepsilon m}{4}%
-d-2m\alpha>-d$, namely $\alpha<\frac{\varepsilon}{8}$. Choose $\alpha
=\frac{\varepsilon}{10}$ in the sequel. We have $W^{\frac{\varepsilon}{10}%
,2m}\left(  D\right)  \subset C\left(  D\right)  $ for $2m\frac{\varepsilon
}{10}>d$, namely for $m>\frac{5d}{\varepsilon}$. Choose for the sequel $m$
equal to the smallest integer such that $m>\frac{5d}{\varepsilon}$,
$m\geq2\vee k_{0}$  with $k_{0}=4$. Then%
\[
E\left[  \left\Vert Z_{t}^{\lambda}\right\Vert _{L^{\infty}\left(  D\right)
}^{2m}\right]  \leq C_{m}^{\prime}\left(  \widetilde{C}_{m}\left(
\lambda\right)  +C_{m}\left(  \lambda\right)  \cdot C_{m,\alpha,\varepsilon
,d}\right)
\]
for some constant $C_{m}^{\prime}>0$. This implies assumption
(\ref{assumption 4.0}). Finally, assumption (\ref{assumption 4}) follows from
$\lim_{\lambda\rightarrow\infty}C_{1}\left(  \lambda\right)  =0$,
$\lim_{\lambda\rightarrow\infty}\widetilde{C}_{1}\left(  \lambda\right)  =0$
and Chebyshev inequality.
\end{proof}
\begin{remark}
\em We think that in this case (\ref{assumption 4}) also follows from Remark \ref{r0}. But we could not find a suitable reference for the condition on $e^{tA},\;t\ge 0,$ in Remark \ref{r0}. Therefore, we have given a direct proof of   (\ref{assumption 4}) above.

\end{remark}

{\bf Acknowledgements}. Support by the De  Giorgi Centre and the DFG through SFB 701 is gratefully acknowledged. The last named author would also like to thank the Scuola Normale Superiore   and the University of Pisa for the support and hospitality during several very pleasant visits during which large parts of this work was done.


\begin{thebibliography}{99}                                                                                               
\bibitem {AlbFer}S. Albeverio, B. Ferrario, Uniqueness of solutions of the
stochastic Navier-Stokes equation with invariant measure given by the
enstrophy, \textit{Ann. Probab.} \textbf{32} (2004), no. 2, 1632--1649.

\bibitem {AlbFer2}S. Albeverio, B. Ferrario, \textit{Some methods of infinite
dimensional analysis in hydrodynamics: an introduction}, SPDE in hydrodynamic:
recent progress and prospects, 1--50, Lecture Notes in Math., 1942, Springer,
Berlin, 2008.

\bibitem{B07} V.I. Bogachev, Measure Theory, vol. 2, Springer 2007.

\bibitem {BDR} V.I. Bogachev, G. Da Prato, M. R\"{o}ckner, Parabolic equations
for measures on infinite-dimensional spaces (Russian) \textit{Dokl. Akad.
Nauk} \textbf{421} (2008), no. 4, 439--444; translation in \textit{Dokl.Math.} \textbf{78} (2008), no. 1, 544--549.

\bibitem {BDR2} V.I. Bogachev, G. Da Prato, M. R\"{o}ckner, Existence and
uniqueness of solutions for Fokker-Planck equations on Hilbert spaces,
\textit{J. Evol. Equ.} \textbf{10} (2010), 487-509.

\bibitem{BDR3} V.I. Bogachev, G. Da Prato, M. R\"{o}ckner,  Uniqueness for solutions of Fokker--Planck equations on infinite dimensional spaces, \textit{Communications in Partial Differential Equations}, \textbf{36}, 925--939, 2011.


\bibitem{BDR11} V.I. Bogachev, G. Da Prato, M. R\"{o}ckner,   Existence results for Fokker-Planck equations in Hilbert spaces. Seminar on Stochastic Analysis, Random Fields and Applications VI,  Progress in Probab., {\bf  63}, 2011, 23--35.
 
\bibitem {BDRRS13} V.I. Bogachev, G. Da Prato, M. R\"{o}ckner,  S. Shaposhnikov, Analytic approach to infinite dimensional continuity and Fokker--Planck equations, CRC 701, Preprint 2013.


\bibitem {CrauelFlandoli}H. Crauel, F. Flandoli, Attractors for random
dynamical systems, \textit{Probab. Theory Related Fields} \textbf{100} (1994),
no. 3, 365--393.

\bibitem {DaPDeb2D} G. Da Prato, A. Debussche, Two-dimensional Navier-Stokes
equations driven by a space-time white noise, \textit{J. Funct. Anal.}
\textbf{196} (2002), no. 1, 180--210. 

\bibitem {DaPDeb}G. Da Prato, A. Debussche, Ergodicity for the 3D stochastic
Navier-Stokes equations, \textit{J. Math. Pures Appl.} (9) \textbf{82} (2003),
no. 8, 877--947.

\bibitem {DaPrato-Zabczyk}G. Da Prato, J. Zabczyk, \textit{Stochastic
Equations in Infinite Dimensions}, Cambridge University Press, Cambridge 1992.

\bibitem {FlaNodea}F. Flandoli, Dissipativity and invariant measures for
stochastic Navier-Stokes equations, \textit{NoDEA Nonlinear Differential
Equations Appl.} \textbf{1} (1994), no. 4, 403--423. 

\bibitem {FlaJFA}F. Flandoli, Irreducibility of the 3-D stochastic
Navier-Stokes equation, \textit{J. Funct. Anal.} \textbf{149} (1997), no. 1, 160--177.

\bibitem {FlaCIME} F. Flandoli, \textit{An introduction to 3D stochastic fluid
dynamics}. SPDE in hydrodynamic: recent progress and prospects, 51--150,
Lecture Notes in Math., 1942, Springer, Berlin, 2008.
 

\bibitem {FlaGat} F. Flandoli, D. Gatarek, Martingale and stationary
solutions for stochastic Navier-Stokes equations, \textit{Probab. Theory
Related Fields} \textbf{102} (1995), no. 3, 367--391. 

\bibitem {FlaRomTrans} F. Flandoli, M. Romito, Partial regularity for the
stochastic Navier-Stokes equations, \textit{Trans. Amer. Math. Soc.}
\textbf{354} (2002), no. 6, 2207--2241.

\bibitem {FlaRom} F. Flandoli, M. Romito, Markov selections for the 3D
stochastic Navier-Stokes equations, \textit{Probab. Theory Related Fields
}\textbf{140} (2008), no. 3-4, 407--458.

\bibitem{PR}  C. Pr\^evot and M. R\"ockner,  A concise course on stochastic partial differential equations.
 Lecture Notes in Mathematics, no.  1905, Springer 2007.


\bibitem{RZZ13} M. R\"ockner, R. Zhu, X. Zhu, A note on stochastic semilinear equations and their associate Fokker--Planck equations, CRC 701- Preprint 2013.

\bibitem {Tem}R. Temam, \textit{Navier-Stokes equations and nonlinear
functional analysis}. CBMS-NSF Regional Conference Series in Applied
Mathematics, 41. Society for Industrial and Applied Mathematics (SIAM),
Philadelphia, PA, 1983.
\end{thebibliography}
\end{document}